\documentclass{article}

\usepackage[english]{babel}

\usepackage{amsmath,amssymb,setspace}
\usepackage{amsthm}
\usepackage[english]{babel}
\usepackage{amsfonts}
\usepackage{calligra}
\usepackage[T1]{fontenc}
\usepackage{xstring}
\usepackage{graphics}
\usepackage[numbers,sort&compress]{natbib}
\usepackage{multicol}
\date{}
\usepackage{longtable}
\usepackage{float}
\usepackage[margin=0.7 in]{geometry}
\usepackage{graphicx}
\usepackage{mathtools}
\usepackage{amsmath}
\usepackage{subcaption}
\usepackage{amsmath, amssymb, setspace}
\usepackage{amsthm}
\usepackage[english]{babel}
\usepackage{amsfonts}
\usepackage{calligra}
\usepackage[T1]{fontenc}
\usepackage{xstring}

\usepackage{amssymb}
\usepackage{amsthm}
\usepackage{amsmath}
\usepackage{lipsum}
\newtheorem{definition}{Definition}
\theoremstyle{plain}
\theoremstyle{definition}
\theoremstyle{remark}
\newtheorem{theorem}{Theorem}

\newtheorem{remark}{Remark}
\newtheorem{example}{Example}

\title{Fractal Calculus to Derive Fractal Frenet Equations for Fractal Curves}
\author{Alireza Khalili Golmankhaneh$^1$, Palle E. T. J{\o}rgensen $^2$, Dimiter Prodanov $^3$\\
$^1$ Department of Physics, Urmia Branch,\\ Islamic Azad University, Urmia 63896,West Azerbaijan,  Iran\\
alirezakhalili2002@yahoo.co.in\\
$^2$ Department of Mathematics,The University of Iowa \\Iowa City, IA
52242-1419, USA\\
palle-jorgensen@uiowa.edu\\
$^3$ ITSDP, IICT, Bulgarian Academy of Sciences, Sofia, Bulgaria.\\
EHS and NERF, Interuniversity Microelectronics Center (Imec), 3001 Leuven, Belgium\\
dimiter.prodanov@imec.be\\
}

\begin{document}

\maketitle

\let\thefootnote\relax
\footnote{Corresponding author Alireza Khalili Golmankhaneh}
\begin{abstract}
This paper introduces the concept of Fractal Frenet equations, a set of differential equations used to describe the behavior of vectors along fractal curves. The study explores the analogue of arc length for fractal curves, providing a measure to quantify their length. It also discusses fundamental mathematical constructs, such as the analogue of the unit tangent vector, which indicates the curve's direction at different points, and the analogue of curvature vector or fractal curvature vector, which characterizes its curvature at various locations. The concept of torsion, describing the twisting and turning of fractal curves in three-dimensional space, is also explored. Specific examples, like the fractal helix and the fractal snowflake, illustrate the application and significance of the Fractal Frenet equations.
\end{abstract} 

\section{Introduction}
Fractal geometry has emerged as a powerful mathematical framework for understanding and describing complex natural patterns that exhibit self-similarity across different scales and dimensions \cite{Mandelbro,fraser2020assouad,ma-12,robertson2020baudelaire,Ewqq,Qaqqqqxcs}. These mesmerizing geometric shapes can be found in various natural occurrences, such as snowflakes, Romanesco broccoli, mountain tops, clouds, tree branches, lightning, river deltas, state borders, buildings, and crystals \cite{b-6}. Their captivating property of endless repetition and similarity at different levels of magnification has led to widespread fascination and application in diverse scientific disciplines.

Fractals often possess fractional dimensions, exceeding their topological dimensions, making them indefinitely complicated and challenging to analyze using traditional Euclidean geometry \cite{rosenberg2020fractal,falconer1999techniques}. Their non-differentiability and lack of integrability have pushed the study of fractal curves towards measures like the Hausdorff measure \cite{Qaswet}. As a result, the exploration of fractal curves and their behaviors requires the development of specialized mathematical tools beyond standard calculus \cite{jorgensen2006analysis}.

The significance of fractals extends beyond mathematics and reaches into various scientific domains, including biology, chemistry, earth sciences, physics, and technology \cite{Rwaqqq}. In physics, fractal geometry has been employed to model physical phenomena in non-integer spaces \cite{Welch-5,Shlesinger-6}. Furthermore, fractal patterns have been experimentally detected, from quantum observations to relativistic light aberration effects \cite{AWqq789}.

Through the application of mathematical methodologies such as fractal geometry, probability theory, and analysis, significant strides have been made in understanding complex processes \cite{ma-7,guo2000oscillation}. In particular, the study of heat transfer in fractal materials involves the use of continuous models with non-integer dimensional spaces \cite{tarasov2016heat,ma-5,uchaikin2013fractional,prodanov2020generalized,ma-6,Trifcebook,ma-9}. Measure theory plays a pivotal role in defining derivation and integration on fractal sets, providing a systematic approach to measure and analyze these structures, and facilitating the development of integral concepts tailored to fractals \cite{bishop2017fractals,jiang1998some,Withers,bongiorno2011henstock,bongiorno2015fundamental,bongiorno2018derivatives,giona1995fractal}.

To address the challenges of dealing with fractal curves and to harness their potential applications, researchers have formulated fractal calculus as a generalization of ordinary calculus, designed to incorporate fractals \cite{parvate2009calculus,parvate2011calculus,Alireza-book}. This extension of calculus has paved the way for investigating stochastic processes, differential equations, and anomalous diffusion on fractal structures \cite{khalili2019random,deppman2023fractal,samayoa2020fractal,golmankhaneh2021fractalBro,golmankhaneh2020stochastic,golmankhaneh2018sub,Alireza-Fernandez-1,golmankhaneh2021equilibrium,golmankhaneh2023fuzzification}.

Focusing on nonstandard Lagrangians as their origin, the study delved into Finsler-Randers manifolds and fractal electrodynamics, revealing their quadratic damping geodesic characteristics \cite{el2022nonstandard,ELNABULSI2022112329}.

It is demonstrated that there exist Hilbert spaces, originating from the Hausdorff measures, that permit the existence of multiresolution wavelets \cite{dutkay2006wavelets}.

Moreover, fractal Fourier and Laplace transforms of local and non-local fractal derivatives have been introduced and utilized to solve fractal differential equations \cite{golmankhaneh2019sumudu,Fourier1,Alireza-book}. These mathematical developments have opened up new avenues for understanding the behavior of fractal curves and analyzing their intricate properties.
Fractal calculus is applied to fractal interpolation functions and Weierstrass functions, both of which have the potential to exhibit non-differentiable and non-integrable behaviors within the realm of ordinary calculus \cite{gowrisankar2021fractal}.

In this paper, we provide a comprehensive overview of fractal geometry, calculus, and their applications in various scientific disciplines. We review the concept of $F^\alpha$-Calculus on fractal curves, explore the analogue of arc length, unit tangent vector, and fractal curvature vector, and discuss the concept of torsion. Additionally, we introduce the Fractal Frenet equations, a set of differential equations that describe the behavior of specific vectors along a fractal curve. By delving into these mathematical constructs, we aim to shed light on the captivating and enigmatic world of fractal geometry and its relevance in the scientific realm.
\section{Review of $F^{\alpha}$-Calculus on Fractal Curves}
 The adaptation of calculus principles to fractal curves through $F^{\alpha}$-calculus provides a powerful mathematical framework for understanding and analyzing functions on these intricate mathematical constructs. This section highlights the significance of $F^{\alpha}$-calculus in dealing with the complexities of fractal geometry and lays the groundwork for further exploration and application in various scientific and engineering domains.
In the following, we want to explain how to parameterize a self-similar fractal curve in two dimensions \cite{parvate2011calculus}.

Let $\mathfrak{F}$ be a self-similar fractal curve, and let $Q_{i}, i=0,...,n-1$ be linear transformations composed of rotations and scalings. Each $Q_{i}$ can be expressed as:

\begin{equation}\label{yyuuup1}
 Q_{i} = s_{i}\begin{pmatrix} \cos\theta_{i} & -\sin \theta_{i} \\ \sin \theta_{i} & \cos \theta_{i} \end{pmatrix},
\end{equation}
subject to the condition:
\begin{equation}\label{yyuuup2}
\sum_{i=0}^{n-1}Q_{i}(\mathbf{r}) = \mathbf{r},
\end{equation}
where $\mathbf{r}$ is a vector and $0 < s_{i} < 1$ for $i=0,...,n-1$. Then, the fractal curve $\mathfrak{F}$ is defined by the iterative equation:

\begin{equation}\label{yyuuup3}
 S_{j}(\mathbf{r}) = \sum_{i=0}^{j-1}Q_{i}(\mathbf{r}_{0}) + Q_{j}(\mathbf{r}), \quad j=0,...,n-1,
\end{equation}
where $\mathbf{r}_{0}$ is a fixed vector. The limit set of this iterative equation will form a curve since the way $S_{j}$ are constructed from $Q_{i}$. To parameterize the fractal curve $\mathfrak{F}$, consider $\lfloor mt\rfloor$ as the integer part of $mt$, and define:
\begin{equation}\label{yyuuup4}
  \mathbf{u}(t) = \mathbf{u}_{x}(t)\hat{i}+\mathbf{u}_{y}(t)\hat{j}=\sum_{i=0}^{\lfloor mt\rfloor-1}Q_{i}(\mathbf{r}_{0}) + Q_{\lfloor mt\rfloor}(\mathbf{u}(mt-\lfloor mt\rfloor)), \quad 0 \leq t \leq 1,
\end{equation}
which provides a parametric representation of $\mathfrak{F}$.
The parameterization and the measurement concepts described here allow us to analyze and quantify properties of the self-similar fractal curve $\mathfrak{F}$ in two dimensions.
Next, we introduce some definitions related to the fractal curve $\mathfrak{F}$ and its calculus:

\begin{definition}
For a set $\mathfrak{F}$ and a subdivision $P_{[a,b]}$, where $a < b$ within the interval $[a_{0},b_{0}]$, we define:
\[ \sigma^{\alpha}[\mathfrak{F},P] = \sum_{i=0}^{n-1} \frac{|\mathbf{u}(t_{i+1})-\mathbf{u}(t_{i})|^{\alpha}}{\Gamma(\alpha+1)}, \]
where $|.|$ denotes the Euclidean norm on $\mathbb{R}^{n}$ and $P_{[a,b]} = \{a=t_{0},...,t_{n}=b\}$.
\end{definition}
\begin{definition}
For a given $\delta > 0$ and $a_{0} \leq a \leq b \leq b_{0}$, the coarse-grained mass is defined as:
\[ \gamma_{\delta}^{\alpha}(\mathfrak{F},a,b) = \inf_{\{P_{[a,b]}: |P| \leq \delta\}} \sigma^{\alpha}[\mathfrak{F},P], \]
where $|P|$ represents the maximum length of the subintervals in the subdivision $P$.
\end{definition}

\begin{definition}
For $a_{0}\leq a < b < b_{0}$, the mass function $\gamma^{\alpha}(\mathfrak{F},a,b)$ is defined as the limit of the coarse-grained mass $\gamma_{\delta}^{\alpha}(\mathfrak{F},a,b)$ as $\delta$ approaches zero. Mathematically, this can be expressed as:
\begin{equation}\label{hy-vh-52}
\gamma^{\alpha}(\mathfrak{F},a,b) = \lim_{\delta\rightarrow0} \gamma_{\delta}^{\alpha}(\mathfrak{F},a,b).
\end{equation}
The mass function $\gamma^{\alpha}(\mathfrak{F},a,b)$ quantifies the fractal properties of the set $\mathfrak{F}$ within the interval $[a, b]$, and it serves as a crucial measure for characterizing self-similar fractal curves. As the value of $\delta$ becomes smaller, the coarse-grained mass $\gamma_{\delta}^{\alpha}(\mathfrak{F},a,b)$ provides a finer approximation to the actual mass $\gamma^{\alpha}(\mathfrak{F},a,b)$. Taking the limit as $\delta$ tends to zero allows us to obtain the precise mass value for the fractal set within the given interval.
\end{definition}

\begin{definition}
The rise function of a fractal curve $\mathfrak{F}$, denoted as $S_{\mathfrak{F}}^{\alpha}(v)$, is defined as follows:

\begin{equation}
  S_{\mathfrak{F}}^{\alpha}(v)=\left\{
                      \begin{array}{ll}
                        \mathfrak{M}^{\alpha}(\mathfrak{F},p_{0},v), & v\geq p_{0} ; \\
                        -\mathfrak{M}^{\alpha}(\mathfrak{F},v,p_{0}), & v<p_{0},
                      \end{array}
                    \right.
\end{equation}
Here, $v\in [a_{1},a_{2}]$, and $p_{0}=a_{1}$ is an arbitrary and fixed number. The function $S_{\mathfrak{F}}^{\alpha}(v)$ represents the mass of the fractal curve $\mathfrak{F}$ up to the point $v$.\\
The rise function is useful for understanding how the mass of the fractal curve accumulates as we traverse along the curve. It provides a measure of the distribution of mass at different points of the curve and helps in characterizing the scaling behavior of $\mathfrak{F}$ with respect to the mass function $\mathfrak{M}^{\alpha}(\mathfrak{F},a,b)$. The rise function is an important tool in the study of self-similar fractals and their properties.
\end{definition}

\begin{definition}
The staircase function, denoted by $J(\theta)$, is defined as follows:
\begin{equation}\label{nnnnn777}
J(\theta)=S_{\mathfrak{F}}^{\alpha}(\textbf{u}^{-1}(\theta)),~~~ \theta\in \mathfrak{F}.
\end{equation}
Here, $\textbf{u}^{-1}(\theta)$ represents the inverse of the mapping function $\textbf{u}(t)$ for $\theta$ belonging to the fractal set $\mathfrak{F}$. The staircase function $J(\theta)$ provides a way to associate points from the fractal curve $\mathfrak{F}$ with their corresponding mass values represented by the staircase function $S_{\mathfrak{F}}^{\alpha}(t)$.

Since $S_{\mathfrak{F}}^{\alpha}(t)$ is a strictly increasing function, it is invertible, ensuring that the mapping $\textbf{u}^{-1}(\theta)$ is well-defined and one-to-one. This means that each point on the fractal curve $\mathfrak{F}$ is uniquely linked to its corresponding mass value through the staircase function.

The staircase function is a valuable tool in analyzing the relationship between the geometrical properties of the fractal curve and its mass distribution. It allows us to explore the distribution of mass along the fractal curve and study how different parts of the curve contribute to the overall mass of the fractal.
\end{definition}

\begin{definition}
The $\gamma$-dimension of the fractal curve $\mathfrak{F}$, denoted by $\textmd{dim}_{\gamma}(\mathfrak{F})$, can be defined as follows:
\begin{align}\label{iimmnn}
\textmd{dim}_{\gamma}(\mathfrak{F}) &= \inf\{\alpha : \gamma^{\alpha}(\mathfrak{F},a,b) = 0\} \nonumber\\
&= \sup\{\alpha : \gamma^{\alpha}(\mathfrak{F},a,b) = \infty\}.
\end{align}
\end{definition}

The $\gamma$-dimension provides a measure of the scaling behavior of the fractal set $\mathfrak{F}$ with respect to the mass function $\gamma^{\alpha}(\mathfrak{F},a,b)$. It helps in quantifying the self-similarity and the fractal nature of $\mathfrak{F}$ within a specific interval $[a,b]$ in $\mathbb{R}^n$.

\begin{definition}
Let $\mathfrak{F}\subset \mathbb{R}^n$ be a fractal curve, and let $f:\mathfrak{F}\rightarrow \mathbb{R}$ be a real-valued function. Given a point $\theta \in \mathfrak{F}$, a number $l$ is considered to be the limit of $f$ through points of $\mathfrak{F}$ if, for any given positive value $\epsilon$, there exists a corresponding $\delta>0$ such that:
\begin{equation}\label{ddss221}
\theta'\in \mathfrak{F}, ~~~ \text{and} ~~~ |\theta' - \theta|<\delta \quad \Rightarrow \quad |f(\theta') - l|<\epsilon.
\end{equation}
If such a number $l$ exists and satisfies the above condition, it is denoted as the fractal limit of $f(\theta)$ and is represented as:
\begin{equation}\label{zzza955159}
  l=\mathfrak{F}_{-}\lim_{\theta'-\theta}f(\theta)
\end{equation}
This definition implies that the function $f(\theta)$ has a well-defined limit at the point $\theta$ on the fractal curve $\mathfrak{F}$ if, as points $\theta'$ on $\mathfrak{F}$ get arbitrarily close to $\theta$, the corresponding values of $f(\theta')$ get arbitrarily close to $l$. In essence, it extends the concept of a limit from ordinary calculus to functions defined on fractal curves.
\end{definition}

\begin{definition}
A function $f:\mathfrak{F}\rightarrow \mathbb{R}$ is considered to be $\mathfrak{F}$-continuous at a point $\theta\in \mathfrak{F}$ if the following condition holds:
\begin{equation}\label{xxzas}
f(\theta)=\mathfrak{F}_{-}\lim_{\theta'\rightarrow \theta}f(\theta').
\end{equation}
In other words, $f(\theta)$ is $\mathfrak{F}$-continuous at $\theta$ if its value at $\theta$ is equal to the fractal limit of $f(\theta')$ as $\theta'$ approaches $\theta$ within the fractal curve $\mathfrak{F}$. This definition characterizes a specific type of continuity suitable for functions defined on fractal curves.
\end{definition}

\begin{definition}
  For a function $f:\mathfrak{F}\rightarrow \mathbb{R}$ and $t_{1},t_{2}\in [a_{0},b_{0}]$ with $t_{1}\leq t_{2}$, we define the upper and lower limits of $f$ over the interval $C(t_{1},t_{2})$ as follows:
\begin{equation}\label{yuolmnbvx74159-9}
M[f,C(t_{1},t_{2})]=\sup_{\theta \in C(t_{1},t_{2})}f(\theta),
\end{equation}
\begin{equation}\label{rrrvx74159-9}
m[f,C(t_{1},t_{2})]=\inf_{\theta \in C(t_{1},t_{2})}f(\theta).
\end{equation}
\end{definition}

\begin{definition}
  Let $S_{\mathfrak{F}}^{\alpha}(t)$ be finite for $t\in [a,b]\subset[a_{0},b_{0}]$. Consider a subdivision $P$ of $[a,b]$ with points ${t_{0},...,t_{n}}$. The upper and lower $\mathfrak{F}^{\alpha}$-sums for the function $f$ over the subdivision $P$ are given respectively by:
\begin{equation}\label{uujji-741}
U^{\alpha}[f,\mathfrak{F},P]=\sum_{i=0}^{n-1}M[f,C(t_{i},t_{i+1})]\left[S_{\mathfrak{F}}^{\alpha}(t_{i+1})-S_{\mathfrak{F}}^{\alpha}(t_{i})\right],
\end{equation}
\begin{equation}\label{uuuu44jji-741}
L^{\alpha}[f,\mathfrak{F},P]=\sum_{i=0}^{n-1}m[f,C(t_{i},t_{i+1})]\left[S_{\mathfrak{F}}^{\alpha}(t_{i+1})-S_{\mathfrak{F}}^{\alpha}(t_{i})\right].
\end{equation}
It is evident from Equations \eqref{uujji-741} and \eqref{uuuu44jji-741} that
\begin{equation}\label{hhh7774114-7}
U^{\alpha}[f,\mathfrak{F},P]\geq L^{\alpha}[f,\mathfrak{F},P].
\end{equation}
\end{definition}

\begin{definition}
  Let $\mathfrak{F}$ be such that $S_{\mathfrak{F}}^{\alpha}$ is finite on $[a,b]$. For a function $f\in B(\mathfrak{F})$, the lower and upper $\mathfrak{F}^{\alpha}$-integrals of $f$ on the section $C(a,b)$ are defined as follows:
\begin{equation}\label{ppo96-4}
\underline{\int_{C(a,b)}}f(\theta)d_{\mathfrak{F}}^{\alpha}\theta= \sup_{P_{[a,b]}} L^{\alpha}[f,\mathfrak{F},P],
\end{equation}
\begin{equation}\label{dddaaaqqq55}
\overline{ \int_{C(a,b)}}f(\theta)d_{\mathfrak{F}}^{\alpha}\theta= \inf_{P_{[a,b]}} U^{\alpha}[f,\mathfrak{F},P].
\end{equation}
These integrals represent the supremum and infimum, respectively, of the lower and upper $\mathfrak{F}^{\alpha}$-sums of $f$ over all possible subdivisions of $[a,b]$. The lower integral captures the largest sum of function values over the subdivisions, while the upper integral captures the smallest sum of function values over the subdivisions.
\end{definition}

\begin{definition}
  If $f\in B(\mathfrak{F})$, we say that $f$ is $\mathfrak{F}^{\alpha}$-integrable on $C(a,b)$ if the lower and upper $\mathfrak{F}^{\alpha}$-integrals of $f$ on $C(a,b)$ are equal, i.e.,
\begin{equation}\label{pplll}
\underline{\int_{C(a,b)}}f(\theta)d_{\mathfrak{F}}^{\alpha}\theta=\overline{ \int_{C(a,b)}}f(\theta)d_{\mathfrak{F}}^{\alpha}\theta,
\end{equation}
and the common value is referred to as the $\mathfrak{F}^{\alpha}$-integral:
\begin{equation}\label{552plmn}
\int_{C(a,b)}f(\theta)d_{\mathfrak{F}}^{\alpha}\theta.
\end{equation}
This integral represents the sum of function values along the fractal curve $\mathfrak{F}$ on the section $C(a,b)$ with respect to the $\mathfrak{F}^{\alpha}$-measure. If this common value exists, then $f$ is considered $\mathfrak{F}^{\alpha}$-integrable on $C(a,b)$.
\end{definition}

\begin{definition}\label{Defoip}
Consider a fractal curve $\mathfrak{F}$. The $\mathfrak{F}^{\alpha}$-derivative of a function $f$ at a point $\theta\in \mathfrak{F}$ is defined as follows:
\begin{equation}\label{frw-98i}
\frac{d^{\alpha}_{\mathfrak{F}}f}{d^{\alpha}_{\mathfrak{F}}\theta}=D_{\mathfrak{F},\theta}^{\alpha}f(\theta)=\mathfrak{F}-\lim_{\theta'\rightarrow\theta}
\frac{f(\theta')-f(\theta)}{J(\theta')-J(\theta)},
\end{equation}
provided the limit exists. Here, $J(\theta)$ represents the staircase function associated with the fractal curve $\mathfrak{F}$ as defined earlier. This derivative extends the concept of ordinary differentiation to functions defined on fractal curves, taking into account the non-smooth and self-replicating nature of these curves.
\end{definition}

\begin{theorem}
  Suppose $f\in B(\mathfrak{F})$ is an $\mathfrak{F}$-continuous function on the section $C(a,b)$, and let $g:\mathfrak{F}\rightarrow \mathbb{R}$ be defined as follows:
\begin{equation}\label{sssaqza3}
g(\textbf{u}(t))=\int_{C(a,t)}f(\theta)d_{\mathfrak{F}}^{\alpha}\theta,
\end{equation}
for all $t\in [a,b]$, where $\textbf{u}(t)$ is the parameterization of the fractal curve $\mathfrak{F}$ given earlier. Then, it follows that:
\begin{equation}\label{rreq11}
D_{\mathfrak{F}}^{\alpha}g(\theta)=f(\theta).
\end{equation}
In other words, the $\mathfrak{F}^{\alpha}$-derivative of the function $g$ is equal to the original function $f$ on the fractal curve $\mathfrak{F}$. This result highlights the connection between the fractal integral and derivative, indicating that integration and differentiation can be understood in the context of fractal curves.
\end{theorem}

\begin{theorem}
  Suppose $f:\mathfrak{F}\rightarrow \mathbb{R}$ is an $\mathfrak{F}^{\alpha}$-differentiable function, and $h:\mathfrak{F}\rightarrow \mathbb{R}$ is $\mathfrak{F}$-continuous such that $h(\theta)=D_{\mathfrak{F}}^{\alpha} f(\theta)$. Then, we have:
\begin{equation}\label{hbvcxz12364-951}
\int_{C(a,b)}h(\theta) d_{\mathfrak{F}}^{\alpha}\theta=f(\textbf{u}(b))-f(\textbf{u}(a)).
\end{equation}
This result shows that the integral of the product of an $\mathfrak{F}$-continuous function $h$ and an $\mathfrak{F}^{\alpha}$-differentiable function $f$ over the section $C(t, t')$ of the fractal curve $\mathfrak{F}$ is equal to the difference in the values of the function $f$ at the endpoints of the section. The theorem establishes a relationship between the fractal integral and the $\mathfrak{F}^{\alpha}$-derivative and provides a way to compute the integral using the derivative of the function.
\end{theorem}

\section{Analogue of Arc Length of Fractal Curve }
This section delves into the definition of the analogue of arc length for fractal curves, as well as the concepts of unit tangent vector, fractal curvature, and torsion. These essential mathematical constructs offer unique insights into the intricate properties of fractal curves, allowing us to navigate the dimensional complexities inherent in these mesmerizing geometrical shapes.
Consider a fractal curve denoted by $\textbf{u}=\textbf{u}(t)$, where $t$ lies within the interval $I=[a_{0},b_{0}]$. In the context of fractal geometry, we define the analogue of arc length, denoted by $s_{\mathfrak{F}}$, for this curve as follows:
\begin{equation}\label{oii}
  s_{\mathfrak{F}}=s_{\mathfrak{F}}(t)=\int_{C(t_{0},t)}\left|
\frac{d^{\alpha}_{\mathfrak{F}}\textbf{u}(t)}
{d^{\alpha}_{\mathfrak{F}}t} \right|d^{\alpha}_{\mathfrak{F}}t
\end{equation}
Moreover, we define the derivative of the fractal arc length, denoted by $d^{\alpha}_{\mathfrak{F}}s_{\mathfrak{F}}/
d^{\alpha}_{\mathfrak{F}}t$, as \cite{khalili2023non}:
\begin{equation}
  \frac{d^{\alpha}_{\mathfrak{F}}s_{\mathfrak{F}}}
{d^{\alpha}_{\mathfrak{F}}t}=\left|
\frac{d^{\alpha}_{\mathfrak{F}}\textbf{u}(t)}
{d^{\alpha}_{\mathfrak{F}}t} \right|
\end{equation}
Here, $\alpha$ denotes the fractal order and $t>t_{0}$. It's important to note that the fractal curve $\textbf{u}(t)$ takes on the value $\theta$.
\begin{example}
Let's consider a fractal helix described by the following equation:
\begin{equation}\label{ttttq9}
  \textbf{u}(t)=a\cos(S_{F}^{\alpha}(t))\hat{e}_{1}+
a\sin(S_{F}^{\alpha}(t))\hat{e}_{2}+b S_{F}^{\alpha}(t) \hat{e}_{3}.
\end{equation}
The analogue of arc length for this fractal helix is given by:
\begin{equation}
  s_{\mathfrak{F}}=\int_{C(t_{0},t)}(a^{2}+b^{2})^{1/2}d^{\alpha}_{\mathfrak{F}}t
=(a^{2}+b^{2})^{1/2}S_{F}^{\alpha}(t)\equiv(a^{2}+b^{2})^{1/2}J(\theta)
\end{equation}
where $S_{F}^{\alpha}(0)=0$.

\begin{figure}[H]
  \centering
  \includegraphics[scale=0.6]{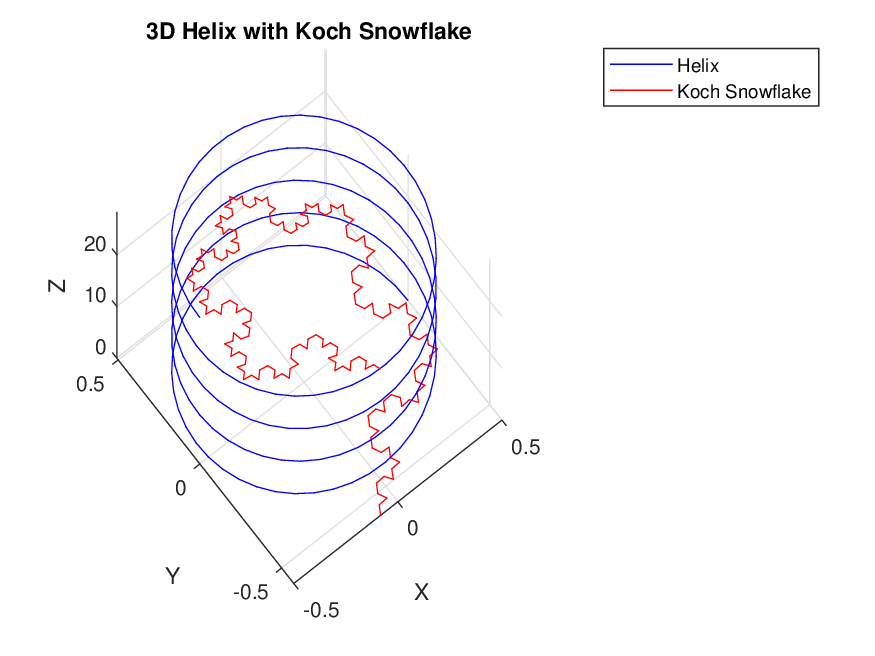}
  \caption{Graph displaying the fascinating fractal helix}\label{jjk4}
\end{figure}

\begin{figure}[H]
  \centering
  \includegraphics[scale=0.6]{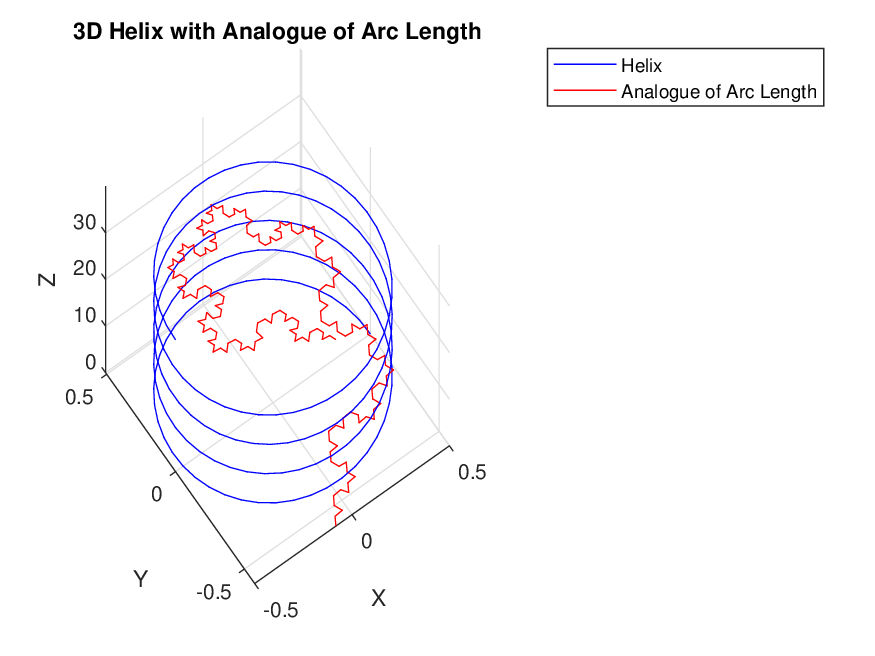}
  \caption{Graph illustrating the analogue of arc length}\label{gt4}
\end{figure}
Figure \ref{jjk4}, intricate illustration reveals the captivating beauty of a coil pattern that replicates itself at different scales, exhibiting intricate geometric intricacies. Embodying the mesmerizing realm of fractal geometry, the helix showcases its distinctive attributes through an infinite sequence of self-replicating coils, resulting in an enchanting array of mesmerizing shapes and forms.\\
Figure \ref{gt4}, visual representation depicts the concept of measuring distances along a curve, capturing the relationship between the length of the curve and the variable defining its position. By plotting the analogue of arc length, we gain valuable insights into how the curve's length changes concerning the parameter or variable that governs its shape. This graphical portrayal offers a clear and intuitive understanding of the concept, enabling us to explore the intricate behavior of curves in relation to their respective arc lengths.

\end{example}
\begin{remark}
We assume the following notations for convenience:
\begin{equation}
 \dot{\textbf{u}}(\theta)=\frac{d^{\alpha}_{\mathfrak{F}}\textbf{u}(\theta)}
{d^{\alpha}_{\mathfrak{F}}\theta},~~~\ddot{\textbf{u}}(\theta)=
\frac{d^{2\alpha}_{\mathfrak{F}}\textbf{u}(t)}
{d^{\alpha}_{\mathfrak{F}}\theta^2},~~~\textbf{u}'(t)=
\frac{d^{\alpha}_{\mathfrak{F}}\textbf{u}(t)}
{d^{\alpha}_{\mathfrak{F}}t},~~~\textbf{u}''(t)=
\frac{d^{2\alpha}_{\mathfrak{F}}\textbf{u}(t)}
{d^{\alpha}_{\mathfrak{F}}t^2}.
\end{equation}
\end{remark}

\begin{definition}

The concept denoting the "analogue of the unit tangent vector" for a fractal curve is as follows: If $\dot{\textbf{u}}(\theta)$ represents the derivative of the fractal curve concerning $\theta$, then $\textbf{t}=\textbf{t}(\theta)=\dot{\textbf{u}}$ is recognized as the equivalent of the unit tangent vector pertaining to the fractal curve.

Alternatively, the tangent vector linked with the fractal curve, labeled as $\textbf{t}_{\mathfrak{F}}$, is articulated as:
\begin{equation}
\textbf{t}_{\mathfrak{F}}=\textbf{t}_{\mathfrak{F}}
(\theta)=\frac{d^{\alpha}_{\mathfrak{F}}\textbf{u}(\theta)}
{d^{\alpha}_{\mathfrak{F}}\theta}.
\end{equation}
This tangent vector conforms to the class $C^{\alpha}$, indicating its qualities of continuity and differentiability.
\end{definition}

\begin{example}
Consider the fractal helix given by the equation:
\begin{equation}
  \textbf{u}(t)=a \cos(S_{\mathfrak{F}}^{\alpha}(t)) \hat{e}_{1}+
b \sin(S_{\mathfrak{F}}^{\alpha}(t)) \hat{e}_{2}+b S_{\mathfrak{F}}^{\alpha}(t)\hat{e}_{3},~~~a\neq0,~b\neq0.
\end{equation}
Then, the derivative of this fractal curve with respect to $t$  is given by:
\begin{equation}
  \frac{d^{\alpha}_{\mathfrak{F}}\textbf{u}(t)}
{d^{\alpha}_{\mathfrak{F}}t}=\frac{1}{\Gamma(\alpha+1)}(-a \sin(S_{\mathfrak{F}}^{\alpha}(t)) \hat{e}_{1}+
b \cos(S_{\mathfrak{F}}^{\alpha}(t)) \hat{e}_{2}+b \chi_{\mathfrak{F}}\hat{e}_{3}),~~~.
\end{equation}
where $\chi_{\mathfrak{F}}$ represents characteristic function.
Consequently, the magnitude of the derivative is given by:
\begin{equation}
\left|\frac{d^{\alpha}_{\mathfrak{F}}\textbf{u}(t)}
{d^{\alpha}_{\mathfrak{F}}t}\right|=\frac{1}{\Gamma(\alpha+1)}(a^2+b^2)^{1/2}.
\end{equation}
Hence, the analogue of the unit tangent vector $\textbf{t}_{\mathfrak{F}}$ for the fractal curve can be expressed as:
\begin{equation}
\textbf{t}_{\mathfrak{F}}=\frac{d^{\alpha}_{\mathfrak{F}}\textbf{u}(\theta)}
{d^{\alpha}_{\mathfrak{F}}\theta}=\frac{\frac{d^{\alpha}_{\mathfrak{F}}\textbf{u}(t)}
{d^{\alpha}_{\mathfrak{F}}t}}{\left|\frac{d^{\alpha}_{\mathfrak{F}}\textbf{u}(t)}
{d^{\alpha}_{\mathfrak{F}}t}\right|}=\frac{1}{(a^2+b^2)^{1/2}}(-a \sin(S_{\mathfrak{F}}^{\alpha}(t)) \hat{e}_{1}+
b \cos(S_{\mathfrak{F}}^{\alpha}(t)) \hat{e}_{2}+b \chi_{\mathfrak{F}}\hat{e}_{3}).
\end{equation}
\end{example}

\begin{definition}
The fractal vector, denoted as $\textbf{k}_{\mathfrak{F}}$, or alternatively, the analogue of curvature vector/fractal curvature vector on $\mathfrak{F}$, is defined as follows:

\begin{equation}
\textbf{k}_{\mathfrak{F}}=
\dot{\textbf{t}}_{\mathfrak{F}}=\frac{d^{\alpha}_{\mathfrak{F}}
\textbf{t}_{\mathfrak{F}}(\theta)}
{d^{\alpha}_{\mathfrak{F}}\theta}=\frac{d^{\alpha}_{\mathfrak{F}}
\textbf{t}_{\mathfrak{F}}}
{d^{\alpha}_{\mathfrak{F}}t}\frac{d^{\alpha}_{\mathfrak{F}}t}
{d^{\alpha}_{\mathfrak{F}}\theta}=\frac{d^{\alpha}_{\mathfrak{F}}
\textbf{t}_{\mathfrak{F}}}
{d^{\alpha}_{\mathfrak{F}}t}\bigg/\frac{d^{\alpha}_{\mathfrak{F}}\theta}
{d^{\alpha}_{\mathfrak{F}}t}=\frac{d^{\alpha}_{\mathfrak{F}}
\textbf{t}_{\mathfrak{F}}}
{d^{\alpha}_{\mathfrak{F}}t}\bigg/\bigg|\frac{d^{\alpha}_{\mathfrak{F}}\textbf{u}}
{d^{\alpha}_{\mathfrak{F}}t}\bigg|
\end{equation}
This vector characterizes the analogue of curvature or fractal curvature on the fractal curve $\mathfrak{F}$.
\end{definition}

\begin{definition}
The magnitude of the fractal curvature is denoted by $\kappa_{\mathfrak{F}}=|\textbf{k}_{\mathfrak{F}}(\theta)|$ and is referred to as the radius of fractal curvature at $\theta$.
Furthermore, its reciprocal, represented by $\rho_{\mathfrak{F}}=\frac{1}{\kappa_{\mathfrak{F}}}$, is also called the radius of fractal curvature at $\theta$.
\end{definition}

\begin{example}
Considering the fractal snowflake with the following  equation:
\begin{equation}
\textbf{u}(t)=(a \cos(S_{\mathfrak{F}}^{\alpha}(t))\hat{e}_{1}+a \sin(S_{\mathfrak{F}}^{\alpha}(t))\hat{e}_{2}),~~~~a>0,
\end{equation}
we find that its fractal derivative with respect to $t$ is given by:
\begin{equation}
  \frac{d^{\alpha}_{\mathfrak{F}}\textbf{u}(t)}
{d^{\alpha}_{\mathfrak{F}}t}=\frac{1}{\Gamma(\alpha+1)}(-a \sin(S_{\mathfrak{F}}^{\alpha}(t))\hat{e}_{1}+a \cos(S_{\mathfrak{F}}^{\alpha}(t))\hat{e}_{2}),~~~
\end{equation}
with a magnitude of:
\begin{equation}
\bigg|\frac{d^{\alpha}_{\mathfrak{F}}\textbf{u}(t)}
{d^{\alpha}_{\mathfrak{F}}t}\bigg|=\frac{a}{\Gamma(\alpha+1)}.
\end{equation}
Hence, the analogue of the unit tangent vector, $\textbf{t}{\mathfrak{F}}$, is given by:
\begin{equation}
  \textbf{t}_{\mathfrak{F}}=(-\sin(S_{\mathfrak{F}}^{\alpha}(t))
\hat{e}_{1}+\cos(S_{\mathfrak{F}}^{\alpha}(t))
\hat{e}_{2})
\end{equation}
and the fractal curvature vector, $\textbf{k}_{\mathfrak{F}}$, can be expressed as:
\begin{equation}
  \textbf{k}_{\mathfrak{F}}=-\frac{1}{a} (\cos(S_{\mathfrak{F}}^{\alpha}(t))\hat{e}_{1}+ \sin(S_{\mathfrak{F}}^{\alpha}(t))\hat{e}_{2})
\end{equation}
\end{example}

\begin{example}
Let us consider the fractal curve described by:
\begin{equation}
  \textbf{u}(t)=S_{\mathfrak{F}}^{\alpha}(t)\hat{ e}_{1}
+\frac{1}{3}S_{\mathfrak{F}}^{\alpha}(t)^3 \hat{ e}_{2}
\end{equation}
Upon calculating the fractal derivative with respect to $t$, we obtain:
\begin{equation}
  \frac{d^{\alpha}_{\mathfrak{F}}\textbf{u}(t)}
{d^{\alpha}_{\mathfrak{F}}t}=\frac{1}{\Gamma(\alpha+1)}(e_{1}+
S_{\mathfrak{F}}^{\alpha}(t)^2\hat{e}_{2}),~~~
\end{equation}
where the magnitude of this derivative is given by:
\begin{equation}
\bigg|\frac{d^{\alpha}_{\mathfrak{F}}\textbf{u}(t)}
{d^{\alpha}_{\mathfrak{F}}t}\bigg|=\frac{1}{\Gamma(\alpha+1)}
(1+S_{\mathfrak{F}}^{\alpha}(t)^4)^{1/2}.
\end{equation}
As a result, the analogue of the unit tangent vector, $\textbf{t}_{\mathfrak{F}}$, is expressed as:
\begin{equation}
  \textbf{t}_{\mathfrak{F}}=
\frac{1}{(1+S_{\mathfrak{F}}^{\alpha}(t)^4)^{1/2}}
(\hat{e}_{1}+
S_{\mathfrak{F}}^{\alpha}(t)^2\hat{e}_{2})
\end{equation}
and the fractal curvature vector, $\textbf{k}_{\mathfrak{F}}$, is given by:
\begin{equation}
  \textbf{k}_{\mathfrak{F}}=-2S_{\mathfrak{F}}^{\alpha}(t)(1+
S_{\mathfrak{F}}^{\alpha}(t)^4)^{-2}(S_{\mathfrak{F}}^{\alpha}(t)^{2}\hat{e}_{1}-\hat{e}_{2})
\end{equation}
\end{example}

\begin{definition}
We define the unit vector in the direction of $\textbf{k}_{\mathfrak{F}}$ as follows:
\begin{equation}\label{dx}
  \textbf{n}_{\mathfrak{F}}=\frac{\textbf{k}_{\mathfrak{F}}(\theta)}
{|\textbf{k}_{\mathfrak{F}}(\theta)|}
\end{equation}
This unit vector points in the same direction as $\textbf{k}_{\mathfrak{F}}$ and has a magnitude of 1, ensuring its normalization.
\end{definition}

\begin{definition}
The unit fractal binormal vector is defined as:
\begin{equation}\label{uytrdx}
  \textbf{b}_{\mathfrak{F}}(\theta)=\textbf{t}_{\mathfrak{F}}(\theta)\times
\textbf{n}_{\mathfrak{F}}(\theta)
\end{equation}
This vector is obtained by taking the cross product of the unit tangent vector $\textbf{t}_{\mathfrak{F}}$ and the unit vector $\textbf{n}_{\mathfrak{F}}(\theta)$ that points in the direction of the fractal curvature $\textbf{k}_{\mathfrak{F}}(\theta)$. The resulting vector, $\textbf{b}_{\mathfrak{F}}(\theta)$, is orthogonal to both $\textbf{t}_{\mathfrak{F}}(\theta)$ and $\textbf{n}_{\mathfrak{F}}(\theta)$ and represents the binormal direction of the fractal curve at $\theta$.
\end{definition}

\begin{example}\label{hhhhh}
Consider the fractal curve described by:
\begin{equation}
  \textbf{u}(t)=a \cos(S_{\mathfrak{F}}^{\alpha}(t)) \hat{e}_{1}+
b \sin(S_{\mathfrak{F}}^{\alpha}(t)) \hat{e}_{2}+b S_{\mathfrak{F}}^{\alpha}(t)\hat{e}_{3},~~~a\neq0,~b\neq0.
\end{equation}
Then, the unit tangent vector $\textbf{t}_{\mathfrak{F}}$, and the fractal curvature vector $\textbf{k}_{\mathfrak{F}}$ are given by:
\begin{equation}
\textbf{t}_{\mathfrak{F}}=(a^2+b^2)^{-1/2}(-a \sin( S_{\mathfrak{F}}^{\alpha}(t))\hat{e}{1}+a
\cos(S_{\mathfrak{F}}^{\alpha}(t))\hat{e}{2}+b\hat{e}{3})
\end{equation}
\begin{equation}
\textbf{k}_{\mathfrak{F}}=-
\frac{a}{a^2+b^2}(\cos(S_{\mathfrak{F}}^{\alpha}(t))
\hat{e}{1}+\sin(S_{\mathfrak{F}}^{\alpha}(t))\hat{e}{2})
\end{equation}
The unit vector $\textbf{n}_{\mathfrak{F}}$ in the direction of $\textbf{k}_{\mathfrak{F}}$ is given by:
\begin{equation}
\textbf{n}_{\mathfrak{F}}=-
(\cos(S_{\mathfrak{F}}^{\alpha}(t))\hat{e}{1}+
\sin(S_{\mathfrak{F}}^{\alpha}(t))\hat{e}{2})
\end{equation}
And the unit fractal binormal vector $\textbf{b}_{\mathfrak{F}}$ is obtained as:
\begin{equation}
\textbf{b}_{\mathfrak{F}}=(a^{2}+b^{2})^{-1/2}(b \sin(S_{\mathfrak{F}}^{\alpha}(t)) \hat{e}{1}-b \cos(S_{\mathfrak{F}}^{\alpha}(t))\hat{e}{2}+a \hat{e}{3})
\end{equation}
The equation of the fractal binormal line at $t=t_{0}$ is then given by:
\begin{align}\label{t9512369}
&y=\textbf{u}(t_{0})+k \textbf{b}_{\mathfrak{F}}(t_{0})\nonumber\\
 & =(a\cos(S_{\mathfrak{F}}^{\alpha}(t_{0}))+k b(a^2+b^2)^{-1/2}\sin(S_{\mathfrak{F}}^{\alpha}(t_{0}))\hat{e}_{1}+
(a\sin(S_{\mathfrak{F}}^{\alpha}(t_{0}))-k b(a^2+b^2)^{-1/2}\cos(S_{\mathfrak{F}}^{\alpha}(t_{0}))\hat{e}_{2}\nonumber\\&+
(bS_{\mathfrak{F}}^{\alpha}(t_{0}))+ak(a^2+b^2)^{-1/2})\hat{e}_{3},~~~-\infty<k<\infty.
\end{align}
\end{example}
\begin{definition}
The fractal torsion of the fractal curve at $\theta$, denoted by $\tau_{\mathfrak{F}}(\theta)$, is defined as follows:
\begin{equation}\label{bnmc}
  \tau_{\mathfrak{F}}(\theta)=-\frac{d^{\alpha}_{\mathfrak{F}}
\textbf{b}_{\mathfrak{F}}}
{d^{\alpha}_{\mathfrak{F}}\theta}\cdot\textbf{n}_{\mathfrak{F}}
\end{equation}
Here, $\textbf{n}_{\mathfrak{F}}$ represents the unit fractal binormal vector, and $d^{\alpha}_{\mathfrak{F}}
\textbf{b}_{\mathfrak{F}}/
d^{\alpha}_{\mathfrak{F}}\theta $ is the fractal derivative of the unit fractal binormal vector with respect to $\theta$.
\end{definition}

\begin{example}
Let's reconsider the fractal helix defined by:
\begin{equation}\label{iiii}
  \textbf{u}(t)=a \cos(S_{\mathfrak{F}}^{\alpha}(t)) \hat{e}_{1}+
b \sin(S_{\mathfrak{F}}^{\alpha}(t)) \hat{e}_{2}+b S_{\mathfrak{F}}^{\alpha}(t)\hat{e}_{3},~~~a\neq0,~b\neq0.
\end{equation}
As obtained previously, the unit fractal binormal vector is given by:
\begin{equation}
  \textbf{b}_{\mathfrak{F}}=(a^2+b^2)^{-1/2}
(b\sin(S_{\mathfrak{F}}^{\alpha}(t)) \hat{e}_{1}-
b\cos(S_{\mathfrak{F}}^{\alpha}(t)) \hat{e}_{2}+a \hat{e}_{3})
\end{equation}
The fractal torsion, denoted by $\tau_{\mathfrak{F}}$, is constant and can be calculated as follows:
\begin{equation}
 \dot{\textbf{b}}_{\mathfrak{F}}= \frac{d^{\alpha}_{\mathfrak{F}}
\textbf{b}_{\mathfrak{F}}}
{d^{\alpha}_{\mathfrak{F}}\theta}=\frac{d^{\alpha}_{\mathfrak{F}}
\textbf{b}_{\mathfrak{F}}}
{d^{\alpha}_{\mathfrak{F}}t}\bigg/\bigg|\frac{d^{\alpha}_{\mathfrak{F}}
\textbf{u}}
{d^{\alpha}_{\mathfrak{F}}t}\bigg|=(a^2+b^2)^{-1}(b \cos(S_{\mathfrak{F}}^{\alpha}(t)) \hat{e}_{1}+
b \sin(S_{\mathfrak{F}}^{\alpha}(t)) \hat{e}_{2})
\end{equation}

\begin{align}
  \tau_{\mathfrak{F}}&=-\frac{d^{\alpha}_{\mathfrak{F}}
\textbf{b}_{\mathfrak{F}}}
{d^{\alpha}_{\mathfrak{F}}\theta}\cdot\textbf{n}_{\mathfrak{F}}\nonumber\\&=
-\frac{1}{\Gamma(\alpha+1)(a^2+b^2)}(b\cos(S_{\mathfrak{F}}^{\alpha}(t))\hat{e}_{1}+
b\sin(S_{\mathfrak{F}}^{\alpha}(t))\hat{e}_{2})\cdot(-
\cos(S_{\mathfrak{F}}^{\alpha}(t))\hat{e}_{1}-
\sin(S_{\mathfrak{F}}^{\alpha}(t))\hat{e}_{2})\nonumber\\&=
\frac{b}{\Gamma(\alpha+1)(a^2+b^2)}
\end{align}
Depending on the sign of $b$, the fractal helix can be classified as a right-handed curve if $b>0$ or a left-handed fractal helix curve if $b<0$.
\end{example}
The following theorem holds for a point on the fractal curve $\textbf{u}(t)$ where $\textbf{k}_{\mathfrak{F}}\neq 0$:

\begin{theorem}
At such a point, the fractal torsion $\tau_{\mathfrak{F}}$ can be expressed as:
\begin{equation}
\tau_{\mathfrak{F}}=\frac{|\textbf{u}'\textbf{u}''\textbf{u}'''|}{|\textbf{u}'\times \textbf{u}'''|^2}
\end{equation}
where $\textbf{u}'$, $\textbf{u}''$, and $\textbf{u}'''$ represent the first, second, and third derivatives of the fractal curve $\textbf{u}(t)$ with respect to $t$.
\end{theorem}

\begin{definition}
The fractal tangent vectors along a fractal curve $C$ give rise to a new fractal curve $\Upsilon$ on the fractal sphere snowflake of radius $1$ centered at the origin, as illustrated in Figure \ref{iiii5}. This curve $\Upsilon$ is referred to as the fractal spherical indicatrix of $\textbf{t}_{\mathfrak{F}}$. If we consider $\textbf{u}=\textbf{u}(\theta)$ as a natural parameterization of the fractal curve $C$, then $\textbf{u}{1}=\textbf{t}_{\mathfrak{F}}(\theta)=\dot{\textbf{u}}(\theta)$ represents the fractal curve $\Upsilon$.
\begin{figure}[H]
  \centering
  \includegraphics[scale=0.6]{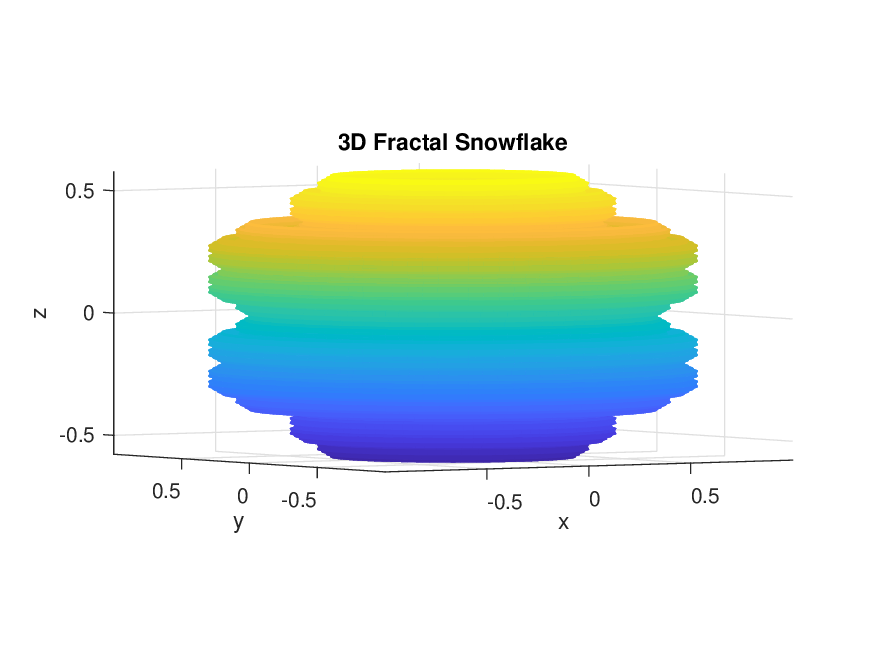}
  \caption{Graph of the 3D fractal snowflake}\label{iiii5}
\end{figure}
\end{definition}
\begin{example}
Consider the fractal helix curve described by:
\begin{equation}
\textbf{u}(t)=a \cos(S_{\mathfrak{F}}^{\alpha}(t)) \hat{e}{1}+
b \sin(S{\mathfrak{F}}^{\alpha}(t)) \hat{e}{2}+b S{\mathfrak{F}}^{\alpha}(t)\hat{e}_{3},~~a>0, b\neq0.
\end{equation}
We find the following properties for the fractal helix curve and its fractal spherical indicatrix:
\begin{align}
\textbf{t}_{\mathfrak{F}}&=(a^2+b^2)^{-1/2}(-a \sin( S_{\mathfrak{F}}^{\alpha}(t))\hat{e}_{1}+
a\cos(S_{\mathfrak{F}}^{\alpha}(t))\hat{e}_{2}+b\hat{e}_{3})\nonumber\\
\textbf{n}_{\mathfrak{F}}&=
-(\cos(S_{\mathfrak{F}}^{\alpha}(t))
\hat{e}_{1}+\sin(S_{\mathfrak{F}}^{\alpha}(t))\hat{e}_{2})\nonumber\\
\textbf{b}_{\mathfrak{F}}&=(a^2+b^2)^{-1/2}
(b\sin(S_{\mathfrak{F}}^{\alpha}(t)) \hat{e}_{1}-
b\cos(S_{\mathfrak{F}}^{\alpha}(t)) \hat{e}_{2}+a \hat{e}_{3})
\end{align}
Note that all these vectors have constant components with respect to $\hat{e}_{3}$, resulting in fractal snowflakes about the $x_{3}$ axis in their spherical images.
The radii of fractal curvature for the fractal spherical indicatrix of $\textbf{t}_{\mathfrak{F}}$, $\textbf{n}_{\mathfrak{F}}$, and $\textbf{b}_{\mathfrak{F}}$ are respectively given by:
\begin{equation}
\varrho_{\textbf{t}_{\mathfrak{F}}}=\frac{a^{2}}{(a^2+b^2)^{1/2}},~~~
\varrho_{\textbf{n}_{\mathfrak{F}}}=1,~~~
\varrho_{\textbf{b}_{\mathfrak{F}}}=\frac{b^{2}}{(a^2+b^2)^{1/2}}
\end{equation}
\end{example}

\section{Fractal Frenet Equation}
In this section, we introduce the analogue of the Frenet equations for a fractal curve. The Fractal Frenet equation is a set of differential equations that describe the behavior of certain vectors along a fractal curve. The Fractal Frenet equations can be expressed as follows:
\begin{theorem}
Along a curve $\textbf{u}=\textbf{u}(\theta)$, the vectors $\textbf{t}_{\mathfrak{F}}$, $\textbf{n}_{\mathfrak{F}}$, and $\textbf{b}_{\mathfrak{F}}$ satisfy the following equations, which are known as the Serret-Frenet equations of the fractal curve:
\begin{align}
  \dot{\textbf{t}}_{\mathfrak{F}}&=\kappa_{\mathfrak{F}}
\textbf{n}_{\mathfrak{F}}\nonumber\\
\dot{\textbf{n}}_{\mathfrak{F}}&=-\kappa_{\mathfrak{F}}
\textbf{t}_{\mathfrak{F}}+\tau_{\mathfrak{F}}\textbf{b}_{\mathfrak{F}}\nonumber\\
\dot{\textbf{b}}_{\mathfrak{F}}&=-\tau_{\mathfrak{F}}\textbf{n}_{\mathfrak{F}}
\end{align}
These equations describe the evolution of the fractal tangent vector $\textbf{t}_{\mathfrak{F}}$, normal vector $\textbf{n}_{\mathfrak{F}}$, and binormal vector $\textbf{b}_{\mathfrak{F}}$ along the fractal curve.
\end{theorem}
\begin{proof}
To prove the second equation, we begin by fractal differentiating $\textbf{n}_{\mathfrak{F}}=\textbf{b}_{\mathfrak{F}}\times \textbf{t}_{\mathfrak{F}}$. This yields:
\begin{equation}
\dot{\textbf{n}}_{\mathfrak{F}} = \dot{\textbf{b}}_{\mathfrak{F}}\times \textbf{t}_{\mathfrak{F}} + \textbf{b}_{\mathfrak{F}}\times \dot{\textbf{t}}_{\mathfrak{F}} = -\tau_{\mathfrak{F}}(\textbf{n}_{\mathfrak{F}}\times \textbf{t}_{\mathfrak{F}}) + \textbf{b}_{\mathfrak{F}} \times (\kappa_{\mathfrak{F}}\textbf{n}_{\mathfrak{F}})
\end{equation}
Since $\textbf{n}_{\mathfrak{F}}\times \textbf{t}_{\mathfrak{F}}$ is orthogonal to $\textbf{b}_{\mathfrak{F}}$, the cross product is zero. So, the equation simplifies to:
\begin{equation}
\dot{\textbf{n}}_{\mathfrak{F}} = \kappa_{\mathfrak{F}}(-\textbf{t}_{\mathfrak{F}}) = -\kappa_{\mathfrak{F}}\textbf{t}_{\mathfrak{F}}
\end{equation}
Thus, we have obtained the second equation of the Serret-Frenet equations for the fractal curve.
\end{proof}
\begin{remark}
A noteworthy observation is that the Frenet equations can be conveniently written in matrix form as:

\begin{align}
  \dot{\textbf{t}}_{\mathfrak{F}}&=0\textbf{t}_{\mathfrak{F}}+\kappa_{\mathfrak{F}}
\textbf{n}_{\mathfrak{F}}+0\textbf{b}_{\mathfrak{F}}\nonumber\\
\dot{\textbf{n}}_{\mathfrak{F}}&=-\kappa_{\mathfrak{F}}
\textbf{t}_{\mathfrak{F}}+\tau_{\mathfrak{F}}\textbf{b}_{\mathfrak{F}}  \nonumber\\
\dot{\textbf{b}}_{\mathfrak{F}}&=-\tau_{\mathfrak{F}}\textbf{n}_{\mathfrak{F}}
\end{align}
This allows us to compactly represent the coefficients of $\textbf{t}{\mathfrak{F}}$, $\textbf{n}{\mathfrak{F}}$, and $\textbf{b}_{\mathfrak{F}}$ in a matrix:
\begin{equation}
  \left(
  \begin{array}{ccc}
    0 & \kappa_{\mathfrak{F}} & 0 \\
    -\kappa_{\mathfrak{F}} & 0 & \tau_{\mathfrak{F}} \\
    0 & -\tau_{\mathfrak{F}} & 0 \\
  \end{array}
\right)
\end{equation}
This matrix representation elegantly summarizes the relationships between the tangent, normal, and binormal vectors along the fractal curve.
\end{remark}
\begin{example}
Consider the fractal logarithmic spiral described by the intrinsic equations $\kappa_{\mathfrak{F}}=1/J(\theta)$ and $\tau_{\mathfrak{F}}=0$, where $\theta>0$. We can further simplify the equations as follows:

\begin{equation}
  \frac{d^{\alpha}_{\mathfrak{F}}\phi}
{d^{\alpha}_{\mathfrak{F}}t}=\kappa_{\mathfrak{F}}=1/J(\theta)
\end{equation}
By integrating the above expression, we obtain:
\begin{equation}
\phi=\log(J(\theta))+C_{1}
\end{equation}
where $J(\theta)=\exp(\phi-C_{1})$ and $\kappa_{\mathfrak{F}}=\exp(-(\phi-C_{1}))$. Setting $\phi-\pi/4=\psi$, we can express the vector $\textbf{u}$ as:

\begin{equation}
  \textbf{u}=\int\frac{1}{\kappa_{\mathfrak{F}}}((\cos \phi)\hat{e}_{1}
+(\sin \phi)\hat{e}_{2})d\phi +C_{2}=\frac{1}{\sqrt{2}}\exp(\psi)[\cos(\psi)\hat{e}_{1}+
\sin(\psi)\hat{e}_{2}]
\end{equation}
Here, we have chosen specific values for constants $C_{1}$ and $C_{2}$, specifically $C_{1}=\pi/4$ and $C_{2}=0$, to simplify the expression. This represents the fractal logarithmic spiral in its natural representation.
\end{example}
\section{Conclusion}
The study of fractal curves and the Fractal Frenet equations provides a fascinating and profound exploration into the world of fractal geometry. These insights not only deepen our understanding of complex shapes but also have potential applications in various fields, including mathematics, physics, computer graphics, and beyond. As researchers continue to investigate and analyze fractal curves, we can anticipate further discoveries and applications that will enrich our understanding of the beauty and intricacy of the natural world.\\

\bibliographystyle{elsarticle-num}

\bibliography{Refrancesma2}

\begin{thebibliography}{10}
\expandafter\ifx\csname url\endcsname\relax
  \def\url#1{\texttt{#1}}\fi
\expandafter\ifx\csname urlprefix\endcsname\relax\def\urlprefix{URL }\fi
\expandafter\ifx\csname href\endcsname\relax
  \def\href#1#2{#2} \def\path#1{#1}\fi

\bibitem{Mandelbro}
B.~B. Mandelbrot, The Fractal Geometry of Nature, WH freeman New York, 1982.

\bibitem{fraser2020assouad}
J.~M. Fraser, Assouad dimension and fractal geometry, Vol. 222, Cambridge University Press, 2020.

\bibitem{ma-12}
M.~L. Lapidus, G.~Radunovi{\'{c}}, D.~{\v{Z}}ubrini{\'{c}}, Fractal Zeta Functions and Fractal Drums, Springer International Publishing, 2017.

\bibitem{robertson2020baudelaire}
L.~Robertson, The Baudelaire Fractal, Coach House Books, 2020.

\bibitem{Ewqq}
N.~Lesmoir-Gordon, B.~Rood, Introducing Fractal Geometry, Icon Books, 2000.

\bibitem{Qaqqqqxcs}
C.~Bovill, Fractal Geometry in Architecture and Design, Birkhauser Boston, MA, 1996.

\bibitem{b-6}
M.~F. Barnsley, Fractals Everywhere, Academic Press, 2014.

\bibitem{rosenberg2020fractal}
E.~Rosenberg, Fractal dimensions of networks, Vol.~1, Springer, 2020.

\bibitem{falconer1999techniques}
K.~Falconer, Fractal Geometry: Mathematical Foundations and Applications, John Wiley \& Sons, 2004.

\bibitem{Qaswet}
P.~R. Massopust, Fractal Functions, Fractal Surfaces, and Wavelets, Academic Press, 2017.

\bibitem{jorgensen2006analysis}
P.~E. Jorgensen, Analysis and Probability: Wavelets, Signals, Fractals, Vol. 234, Springer Science \& Business Media, 2006.

\bibitem{Rwaqqq}
T.~Vicsek, M.~Shlesinger, M.~Matsushita (Eds.), Fractals in Natural Sciences, World Scientific, 1994.

\bibitem{Welch-5}
K.~Welch, A Fractal Topology of Time: Deepening into Timelessness, Fox Finding Press, 2020.

\bibitem{Shlesinger-6}
M.~F. Shlesinger, Fractal time in condensed matter, Annu. Rev. Phys. Chem. 39~(1) (1988) 269--290.

\bibitem{AWqq789}
A.~Jadczyk, Quantum Fractals: From Heisenberg's Uncertainty to Barnsley's Fractality, World Scientific, 2014.

\bibitem{ma-7}
M.~T. Barlow, E.~A. Perkins, Brownian motion on the sierpinski gasket, Probab. Theory Rel. 79~(4) (1988) 543--623.

\bibitem{guo2000oscillation}
J.~Guo, The oscillation of the occupation time process of super-brownian motion on sierpinski gasket, Sci. China Math. 43 (2000) 1250--1257.

\bibitem{tarasov2016heat}
V.~E. Tarasov, Heat transfer in fractal materials, International Journal of Heat and Mass Transfer 93 (2016) 427--430.

\bibitem{ma-5}
A.~S. Balankin, A continuum framework for mechanics of fractal materials i: from fractional space to continuum with fractal metric, Eur. Phys. J. B 88~(4) (Apr. 2015).

\bibitem{uchaikin2013fractional}
V.~V. Uchaikin, Fractional Derivatives for Physicists and Engineers, Vol.~2, Springer, 2013.

\bibitem{prodanov2020generalized}
D.~Prodanov, Generalized differentiability of continuous functions, Fractal and Fractional 4~(4) (2020) 56.

\bibitem{ma-6}
V.~E. Tarasov, Fractional Dynamics, Springer Berlin Heidelberg, 2010.

\bibitem{Trifcebook}
T.~Sandev, {\v{Z}}.~Tomovski, Fractional Equations and Models, Springer International Publishing, 2019.

\bibitem{ma-9}
M.~Czachor, Waves along fractal coastlines: From fractal arithmetic to wave equations, Acta Phys. Pol. B 50~(4) (2019) 813.

\bibitem{bishop2017fractals}
C.~J. Bishop, Y.~Peres, Fractals in probability and analysis, Vol. 162, Cambridge University Press, 2017.

\bibitem{jiang1998some}
H.~Jiang, W.~Su, Some fundamental results of calculus on fractal sets, Commun. Nonlinear Sci. Numer. Simul. 3~(1) (1998) 22--26.

\bibitem{Withers}
W.~Withers, Fundamental theorems of calculus for {H}ausdorff measures on the real line, J. Math. Anal. Appl. 129~(2) (1988) 581--595.

\bibitem{bongiorno2011henstock}
D.~Bongiorno, G.~Corrao, et~al., The {H}enstock-{K}urzweil-{S}tieltjes type integral for real functions on a fractal subset of the real line, in: Bollettino di Matematica pura e applicata vol. IV, Vol.~4, Aracne, 2011, pp. 5--16.

\bibitem{bongiorno2015fundamental}
D.~Bongiorno, G.~Corrao, On the fundamental theorem of calculus for fractal sets, Fractals 23~(02) (2015) 1550008.

\bibitem{bongiorno2018derivatives}
D.~Bongiorno, Derivatives not first return integrable on a fractal set, Ric. di Mat. 67~(2) (2018) 597--604.

\bibitem{giona1995fractal}
M.~Giona, Fractal calculus on [0, 1], Chaos Solit. Fractals 5~(6) (1995) 987--1000.

\bibitem{parvate2009calculus}
A.~Parvate, A.~D. Gangal, Calculus on fractal subsets of real line-{I}: Formulation, Fractals 17~(01) (2009) 53--81.

\bibitem{parvate2011calculus}
A.~Parvate, S.~Satin, A.~Gangal, Calculus on fractal curves in $\mathbb{R}^{n}$, Fractals 19~(01) (2011) 15--27.

\bibitem{Alireza-book}
A.~K. Golmankhaneh, Fractal Calculus and its Applications, World Scientific, 2022.

\bibitem{khalili2019random}
A.~K. Golmankhaneh, A.~Fernandez, Random variables and stable distributions on fractal {C}antor sets, Fractal Fract. 3~(2) (2019) 31.

\bibitem{deppman2023fractal}
A.~Deppman, E.~Megias, R.~Pasechnik, Fractal derivatives, fractional derivatives and $ q $-deformed calculus, arXiv preprint arXiv:2305.04633 (2023).

\bibitem{samayoa2020fractal}
D.~Samayoa, L.~Ochoa-Ontiveros, L.~Dami{\'a}n-Adame, E.~Reyes~de Luna, L.~{\'A}lvarez-Romero, G.~Romero-Paredes, Fractal model equation for spontaneous imbibition, Rev. Mex. de Fis. 66~(3) (2020) 283--290.

\bibitem{golmankhaneh2021fractalBro}
A.~K. Golmankhaneh, R.~T. Sibatov, Fractal stochastic processes on thin {C}antor-like sets, Mathematics 9~(6) (2021) 613.

\bibitem{golmankhaneh2020stochastic}
A.~K. Golmankhaneh, C.~Tun{\c{c}}, Stochastic differential equations on fractal sets, Stochastics 92~(8) (2020) 1244--1260.

\bibitem{golmankhaneh2018sub}
A.~K. Golmankhaneh, A.~S. Balankin, Sub-and super-diffusion on {C}antor sets: Beyond the paradox, Phys. Lett. A. 382~(14) (2018) 960--967.

\bibitem{Alireza-Fernandez-1}
A.~K. Golmankhaneh, A.~Fernandez, A.~K. Golmankhaneh, D.~Baleanu, Diffusion on middle-$\xi$ {C}antor sets, Entropy 20~(7) (2018) 504.

\bibitem{golmankhaneh2021equilibrium}
A.~K. Golmankhaneh, K.~Welch, Equilibrium and non-equilibrium statistical mechanics with generalized fractal derivatives: A review, Mod. Phys. Lett. A 36~(14) (2021) 2140002.

\bibitem{golmankhaneh2023fuzzification}
A.~K. Golmankhaneh, K.~Welch, C.~Serpa, P.~E. J{\o}rgensen, Fuzzification of fractal calculus, arXiv preprint arXiv:2302.07641 (2023).

\bibitem{el2022nonstandard}
R.~A. El-Nabulsi, A.~K. Golmankhaneh, Nonstandard and fractal electrodynamics in finsler--randers space, Int. J. Geom. Methods M. (2022) 2250080.

\bibitem{ELNABULSI2022112329}
R.~A. El-Nabulsi, A.~{Khalili Golmankhaneh}, P.~Agarwal, On a new generalized local fractal derivative operator, Chaos Solit. Fractals 161 (2022) 112329.

\bibitem{dutkay2006wavelets}
D.~E. Dutkay, P.~E. Jorgensen, Wavelets on fractals, Rev. Mat. Iberoamericana 22~(1) (2006) 131--180.

\bibitem{golmankhaneh2019sumudu}
A.~K. Golmankhaneh, C.~Tun{\c{c}}, Sumudu transform in fractal calculus, Appl. Math. Comput. 350 (2019) 386--401.

\bibitem{Fourier1}
A.~K. Golmankhaneh, K.~Ali, R.~Yilmazer, M.~Kaabar, Local fractal {F}ourier transform and applications, Comput. Methods Differ. Equ. 10~(3) (2021) 595--607.

\bibitem{gowrisankar2021fractal}
A.~Gowrisankar, A.~K. Golmankhaneh, C.~Serpa, Fractal calculus on fractal interpolation functions, Fractal Fract. 5~(4) (2021) 157.

\bibitem{khalili2023non}
A.~K. Golmankhaneh, K.~Welch, C.~Serpa, P.~E. J{\o}rgensen, Non-standard analysis for fractal calculus, The Journal of Analysis 31 (2023) 1895--1916.

\end{thebibliography}
\end{document}